\documentclass{article}

\usepackage{amsmath, mathrsfs, amsthm, lmodern, amssymb}
\usepackage[all]{xy}

\theoremstyle{plain}
\newtheorem{Thm}{Theorem}[section]
\newtheorem{Lem}[Thm]{Lemma}
\newtheorem{Cor}[Thm]{Corollary}
\newtheorem{Prop}[Thm]{Proposition}
\newtheorem{Fact}[Thm]{Fact}
\newtheorem{Conj}[Thm]{Conjecture}

\theoremstyle{definition}
\newtheorem{Def}[Thm]{Definition}
\newtheorem{Ex}[Thm]{Example}

\newcommand{\nats}{\mbox{\( \mathbb N \)}}
\newcommand{\set}[1]{\{#1\}}

\DeclareMathOperator{\sL}{\mathscr{L}}

\DeclareMathOperator{\bb}{\bar{\mathit{b}}}

\DeclareMathOperator{\ba}{\bar{\mathit{a}}}

\DeclareMathOperator{\sP}{\wp}
\DeclareMathOperator{\sC}{\mathscr{C}}

\DeclareMathOperator{\UP}{ \prod_{\mathit{U} } \mathit{L}}

\DeclareMathOperator{\bw}{\bigwedge}
\DeclareMathOperator{\bv}{\bigvee}

\DeclareMathOperator{\DL}{\textbf{DL}}
\DeclareMathOperator{\CRL}{\textbf{CRL}}
\DeclareMathOperator{\mCRL}{\textbf{mCRL}}
\DeclareMathOperator{\jCRL}{\textbf{jCRL}}
\DeclareMathOperator{\biCRL}{\textbf{biCRL}}
\DeclareMathOperator{\R}{\mathbb{R}}
\DeclareMathOperator{\Q}{\mathbb{Q}}
\DeclareMathOperator{\xL}{\mathit{x}^{\bL}}
\DeclareMathOperator{\yL}{\mathit{y}^{\bL}}
\DeclareMathOperator{\zL}{\mathit{z}^{\bL}}
\DeclareMathOperator{\bL}{\mathbb{A}}
\DeclareMathOperator{\bS}{\mathbb{S}}
\DeclareMathOperator{\sS}{\mathit{s}^{\bS}}
\DeclareMathOperator{\uS}{\mathit{u}^{\bS}}
\DeclareMathOperator{\tS}{\mathit{t}^{\bS}}

\begin{document}

\title{Completely representable lattices}
\author{Robert Egrot and Robin Hirsch}
\date{}
\maketitle

\begin{abstract}
It is known that a lattice is representable as a ring of sets iff the lattice is distributive.
$\CRL$  is the class of  bounded distributive lattices (DLs) which have representations preserving arbitrary joins and meets.  $\jCRL$ is the class of DLs which have representations preserving arbitrary joins, $\mCRL$ is the class of DLs which have representations preserving arbitrary meets, and $\biCRL$ is defined to be $\jCRL \cap \mCRL$.   We prove
\begin{equation*} \CRL\subset \biCRL=\mCRL\cap\jCRL\subset \mCRL\neq\jCRL \subset \DL  \end{equation*}
where the marked inclusions are proper.  

Let $L$ be a DL.  Then $L\in\mCRL$ iff $L$ has a distinguishing set of complete, prime filters.  Similarly, $L\in\jCRL$ iff $L$ has a distinguishing set of completely prime filters, and $L\in\CRL$ iff $L$ has a distinguishing set of complete, completely prime filters.

Each of the classes above is shown to be  \emph{pseudo-elementary} hence closed under ultraproducts.  
The class $\CRL$ is not closed under elementary equivalence, hence it is not elementary.

\end{abstract}

\begin{section}{Introduction}\label{S:intro}
An \emph{atomic representation} $h$ of a Boolean algebra $ B$ is a representation $h\colon B\to\sP(X)$ (some set $X$) where $h(1)=\bigcup \set{h(a)\colon a\mbox{ is an atom of } B}$. It is known  that a representation of a Boolean algebra is a complete representation (in the sense of a complete embedding into a field of sets) if and only if it is an atomic representation  and hence that the class of completely representable Boolean algebras is precisely the class of atomic Boolean algebras, and hence is elementary \cite{HirHod97}. This result is not obvious as the usual definition of a complete representation is thoroughly second order. The purpose of this note is to investigate the possibility of corresponding results for the class of bounded, distributive lattices. The situation is a little more complex in this case as in the absence of Boolean complementation a representation of a (distributive) lattice may be complete with respect to one of the lattice operations but not the other. 

It turns out (theorem \ref{not elem}) that the class $\CRL$ of completely representable bounded, distributive lattices is \emph{not} elementary, however, building on early work in lattice theory by Birkhoff \cite{Bir33}, and Birkhoff and Frink \cite{BirFri48} it is possible to characterise complete representability of a lattice in terms of the existence of certain prime filters (or dually using prime ideals). Using this characterisation an alternative proof of the identification of the completely representable Boolean algebras with the atomic ones is provided. It is also shown that $\CRL$, and the classes of (bounded, distributive) lattices that have representations respecting either or both arbitrary infima and suprema \emph{are} pseudo-elementary, and thus closed under ultraproducts. Using the well known fact that a class is elementary if and only if it is closed under isomorphism, ultraproducts and ultraroots it follows that $\CRL$ is not closed under ultraroots. The question of whether this holds for the other classes of lattices under consideration, and thus whether they are elementary, remains open at this time and is the subject of ongoing investigation.
\end{section}

\begin{section}{Representations of bounded, distributive lattices}
\begin{Def}[Representation]
Let $L$ be a bounded, distributive lattice. A \emph{representation} of $L$ is an embedding $h\colon L\to\sP(X)$ for some set $X$, where $\sP(X)$ is considered as a ring of sets, under the operations of set union and intersection. When such a representation exists we say that $L$ is \emph{representable}.
\end{Def}
For simplicity we shall assume that our representations $h\colon L\to \sP(X)$ are `irredundant', that is, for all $x\in X$ there is some $a\in L$ with $x\in h(a)$. For irredundant representations $h\colon L\to \sP(X)$ the `inverse image' $h^{-1}[x]=\{a\in L\colon  x\in h(a)\}$ of any point $x\in X$ is a prime filter, with closure under finite meets coming from finite meet preservation by the representation, and primality coming from finite join preservation. Upward closure can be derived from either of these preservation properties using the equivalent definitions of the order relation in a lattice. Conversely, any set $K$ of prime filters of $L$ with the property that for every pair $a\neq b \in L$ there exists $f\in K$ with either $a\in f$ and $b\not \in f$ or vice versa determines a representation $h_K\colon L \to \sP(K)$ using $h_K(a)=\{f\in K\colon  a\in f\}$ (note that for $f\in K$ we have $h_K^{-1}[f]=f$). For ease of exposition later we introduce a definition for sets of sets generalising the condition for filters given above.
\begin{Def}[Distinguishing set]
A set $S\subseteq\sP(L)$ is \emph{distinguishing} over $L$ iff for every pair $a\neq b \in L$ there exists $s\in S$ with either $a\in s$ and $b\not \in s$ or vice versa.
\end{Def}
Using this definition we state the results of the preceding discussion as a simple theorem:
\begin{Thm} A bounded distributive lattice $L$ is representable if and only if it has a distinguishing set of prime filters.
\end{Thm}
As a consequence of the prime ideal theorem for distributive lattices we have:
\begin{Thm}
A bounded lattice is representable if and only if it is distributive.
\end{Thm}
Henceforth, all lattices under consideration are bounded and distributive.
We now discuss representations preserving arbitrary meets and/or joins.

\begin{Def}[Meet-complete map]
A lattice map $f\colon L_1\to L_2$ is \emph{meet-complete} if for all $S\subseteq L_1$ where $\bw S$ exists in $L_1$ we have $f(\bw_{L_1} S)=\bw_{L_2} f[S]$.
\end{Def}

A similar definition is made for \emph{join-complete}. When a map is both \emph{meet-complete} and \emph{join-complete} we say it is \emph{complete}. When a bounded, distributive lattice has a meet-complete representation we say it is \emph{meet-completely representable}, and we make similar definitions for join-complete and complete representations. We shall call the class of all bounded, distributive lattices $\DL$, the class of all completely representable lattices $\CRL$, the classes of meet and join-completely representable lattices $\mCRL$ and $\jCRL$ respectively, and the class of lattices with both a meet-complete and a join-complete representation $\biCRL$. 

\begin{Thm}\label{duality}
A lattice $L$ has a meet-complete representation iff its order dual $L^{\delta}$ has a join-complete representation.
\end{Thm}
\begin{proof}
If $h\colon L\to \sP(P)$ is a representation, where $P$ is some distinguishing set of prime filters of $L$, then the map $\bar{h}\colon L^{\delta}\to \sP(P)$, $a\mapsto -h(a)$ is also a representation. If $h$ is meet-complete then by De Morgan $\bar{h}(\bv_{\delta} S) = -h(\bw S) = - \bigcap h[S] = -\bigcap -\bar{h}[S] = \bigcup \bar{h}[S]$ (here `$-$' denotes set theoretic complement).
\end{proof} 

\begin{Def}[Complete ideal/filter]
An ideal $I$ of a lattice $L$ is complete if whenever $\bv S$ exists in $L$ for $S\subseteq I$ then $\bv S\in I$. Similarly a filter $F$ of $L$ is complete if whenever $\bw T$ exists in $L$ for $T\subseteq F$ then $\bw T \in F$. 
\end{Def}

\begin{Def}[Completely-prime ideal/filter]
A prime ideal $I$ of $L$ is completely prime if whenever $\bw T\in I$ for some $T\subseteq L$ then $I\cap T \not= \emptyset$. Similarly, a prime filter $F$ of $L$ is completely prime if whenever $\bv S\in F$ for some $S\subseteq L$ then $F\cap S \neq \emptyset$.
\end{Def}

\begin{Lem}\label{ideal lemma}
If $F$ is a prime filter of $L$ and $I=L\setminus F$ is its prime ideal complement then $F$ is complete iff $I$ is completely prime, and $I$ is complete iff $F$ is completely prime. 
\end{Lem}

\begin{proof}
Using $I=L\setminus F$ we can rewrite the definition of completeness of $I$ as $\bv S \in F \implies F\cap S \not=\emptyset$. Similarly we can write completeness for $F$ as $\bw T \in I \implies T\cap I \not=\emptyset$.
\end{proof}

\begin{Thm}\label{main2}
Let $L$ be a bounded, distributive lattice. Then:
\begin{enumerate}
\item $L$ has a meet-complete representation iff $L$ has a distinguishing set of complete, prime filters,
\item $L$ has a join-complete representation iff $L$ has a distinguishing set of completely-prime filters,
\item $L$ has a complete representation iff $L$ has a distinguishing set of complete, completely-prime filters,
\end{enumerate}
\end{Thm}
\begin{proof}
We prove 1), the rest follows from theorem \ref{duality} and lemma \ref{ideal lemma}. For the left to right implication, let $h$ be an irredundant meet-complete representation of $L$ over some domain $D$.  Since $h$ is an irredundant representation, $\set{h^{-1}[d]\colon d\in D}$ is a distinguishing set of prime filters.  Also, if $S\subseteq h^{-1}[d]$ then $d\in \bigcap h[S]$ and by completeness of $h$ if $\bw S$ exists  then $d\in h(\bw S)$ so $\bw S\in h^{-1}[d]$, hence each of these prime filters is complete.  Conversely, let $K$ be a distinguishing set of complete, prime filters.  Define a map $h\colon L\to\wp(K)$ by $h(l)=\set{p\in K\colon l\in p}$.  Because $K$ is a distinguishing set of prime filters, $h$ is a representation.  Furthermore, since each $p\in K$ is complete, if $S\subseteq L$ and $\bw S$ exists then for all $p\in K,$
\begin{align*}
p\in h(\bw S)&\iff\bw S\in p\\
&\iff S\subseteq p\\
&\iff p\in \bigcap h[S]
\end{align*}
so $h(\bw S)=\bigcap h[S]$ and $h$ is a complete representation.

\end{proof}
In the light of lemma \ref{ideal lemma} it's straightforward to prove an analogous result to theorem \ref{main2} using ideals in place of filters.

We briefly turn our attention to the special case of Boolean algebras.  Recall that a bounded lattice $(L, 0, 1, \wedge, \vee)$  is \emph{complemented} iff for all $s\in L$ there is $s'\in L$ such that $s\vee s'=1$ and $s\wedge s'=0$. Since there can be at most one complement to an element,  we may write $-s$ instead of $s'$.

\begin{Lem}
If $L$ is complemented then its prime filters are precisely its ultrafilters, moreover the following are equivalent:
\begin{enumerate}
\item $U$ is a principal ultrafilter of $L$,
\item $U$ is a complete ultrafilter of $L$,
\item $U$ is a completely-prime ultrafilter of $L$.
\end{enumerate}
\end{Lem}
\begin{proof}
It's easy to see that the ultrafilters of a BA are precisely its prime filters. Clearly $1)\implies 2)$. Let $U$ be an ultrafilter. If $U$ is complete it must contain a non-zero lower bound $s$ and thus be principal (otherwise it would contain the complement of that lower bound, but $s\leq -s\Rightarrow s=0$),  so $2)\implies 1)$. For any $S\subseteq L$ we write $-S$ for $\set{-s\colon s\in S}$.
The infinite De Morgan law for Boolean algebras (see e.g. \cite[section~19]{Sik69}) gives $-\bv S=\bw-S$ so if $U$ is complete then $S\cap U=\emptyset \implies-\bv S \in U \implies \bv S\notin U$, so $2)\implies3)$. Similarly, if $U$ is completely-prime then $\bw S \notin U \implies -\bw S \in U \implies \bv -S \in U \implies -s \in U $ for\ some $s\in S\implies S\not \subseteq U$, so $3)\implies 2)$.
\end{proof}

We have as a corollary the following result (see \cite[corollary~6]{HirHod97} for the equivalence of the first two parts).

\begin{Cor}\label{BooCor}
For a Boolean algebra $B$ the following are equivalent:
\begin{enumerate}
\item $B$ is atomic,
\item $B$ is completely representable,
\item $B$ is meet-completely representable,
\item $B$ is join-completely representable.
\end{enumerate}
\end{Cor}

Turning our attention back to the lattice case we now give some examples to illustrate the relationships between the classes we have defined.
\begin{Ex}\label{exI}\emph{A distributive lattice both meet-completely representable and join-completely representable but not completely representable.}
Let $L=[0,1]\subseteq \R$. Then by taking $\{ [x, 1]\colon x\in L\}$ we obtain a distinguishing set of complete, prime filters, and by taking $\{ (x, 1]\colon x\in L\}$ we obtain a distinguishing set of completely-prime filters.\\
 However, if $F$ is a complete filter of $L$ then $\bw F\in F$ (by completeness properties of $L$ and $F$) and, since $\bw F=\bv \{x\in L\colon x<\bw F\}$, $F$ cannot be completely-prime.  
\end{Ex} 

\begin{Ex}\label{exII}\emph{A distributive lattice neither meet nor join-completely representable.}
In view of corollary \ref{BooCor} we can take any Boolean algebra that fails to be atomic.
\end{Ex}

\begin{Ex}\label{exIII}\emph{A distributive lattice join-completely representable but not meet-completely representable.}
Let $L$ be the lattice $(\overline\nats\times\overline\nats) \cup\set{0}$ shown in figure~\ref{fig:w2}, where $\overline\nats$ is the set of non-positive integers under the usual ordering and the element $0$ is a lower bound for the whole lattice. Then $L$ has no complete, prime filters, but all its filters are completely-prime, hence by theorem \ref{main2} it has a join-complete representation but no meet-complete representation.
\begin{figure}
\[
\xymatrix@dr{\bullet\ar@{-}[r]\ar@{-}[d]   &    \bullet\ar@{-}[r]\ar@{-}[d]    &    \bullet\ar@{-->}[r]\ar@{-}[d]     &    \\
\bullet\ar@{-}[r]\ar@{-}[d]   &   \bullet\ar@{-}[r]\ar@{-}[d]   &    \bullet\ar@{-->}[r]\ar@{-}[d]    &\\
\bullet\ar@{-}[r]\ar@{-->}[d]&\bullet\ar@{-}[r]\ar@{-->}[d]&\bullet\ar@{-->}[r]\ar@{-->}[d]   &\\
&&&\bullet 0
}
\]
\caption{\label{fig:w2}The lattice  $(\overline\nats\times\overline\nats) \cup\set{0}$}
\end{figure}
\end{Ex}

Examples \ref{exI}, \ref{exII} and \ref{exIII} (and its dual) give us the following:

\begin{equation}\tag{$\dag$} \label{dagger} \CRL\subset \biCRL=\mCRL\cap\jCRL\subset \mCRL\neq\jCRL \subset \DL \end{equation}

There is a relationship between the existence of types of complete representation and the join and meet-densities of the sets of join and meet-irreducibles of $L$, which we make precise in the following proposition.

\begin{Prop}\label{irred}
Let $L$ be a bounded, distributive lattice. Define $J(L)$ and $M(L)$ to be the sets of join-irreducible and meet-irreducible elements of $L$ respectively, and $J^\infty_p(L)$ and $M^\infty_p(L)$ to be the sets of completely join/ meet-primes of $L$ respectively, then:
\begin{enumerate}
\item If the set $J(L)$ is join-dense in $L$ then $L$ has a meet-complete representation, dually if the set $M(L)$ is meet-dense in $L$ then $L$ has a join-complete representation. When $L$ is complete then if $L$ has a meet/join-complete representation the sets $J(L)/M(L)$ are join/meet-dense in $L$. 
\item If either $J^\infty_p(L)$ is join-dense in $L$ or $M^\infty_p(L)$ is meet-dense in $L$ then $L$ has a complete representation. When $L$ is complete it is also true that whenever $L$ has a complete representation $J^\infty_p(L)$ and $M^\infty_p(L)$ are join and meet-dense in $L$ respectively. 
\end{enumerate}
\end{Prop}
\begin{proof}

For the first part of 1, we just take the sets of principal filters/ideals generated by the join/meet-irreducibles respectively, for the second we note that the generator of each filter/ideal must be join/meet-irreducible. For the first part of 2 we note that if we take the sets of principal filters/ideals generated by $J^\infty_p(L)$ and $M^\infty_p(L)$ respectively we obtain distinguishing sets of completely-prime filters/ideals, and for the second part the generator of each filter/ideal will be completely join/meet-irreducible.
\end{proof}

Note that the full converses to proposition \ref{irred} (i.e. when $L$ is not complete) do not hold, so e.g. in a completely representable lattice $L$, $J^{\infty}_p(L)$ need not be join-dense, as the following example illustrates.

\begin{Ex} \label{ex2} $L$ is the lattice with domain $(\overline\nats\times\overline\nats)\cup\nats$ as shown in figure~\ref{fig:w3}, where $\overline\nats$ is the set of  non-positive integers under their usual ordering and each element of $\nats$ is less than each element of $(\overline\nats\times\overline\nats)$.  For $-n\in \overline\nats$, the set $[-n, 0]\times\overline\nats$ is a complete, completely-prime filter (with no infimum) and similarly $\overline\nats\times [-n, 0]$ is also complete, completely-prime.
Hence $L$ has a distinguishing set of complete, completely-prime filters but $J^{\infty}_p(L)=J(L)=\nats$ is not join dense in $L$.   

\begin{figure}
\[\begin{array}{c}
\xymatrix@dr{\bullet\ar@{-}[r]\ar@{-}[d]   &    \bullet\ar@{-}[r]\ar@{-}[d]    &    \bullet\ar@{-->}[r]\ar@{-}[d]     &    \\
\bullet\ar@{-}[r]\ar@{-}[d]   &   \bullet\ar@{-}[r]\ar@{-}[d]   &    \bullet\ar@{-->}[r]\ar@{-}[d]    &\\
\bullet\ar@{-}[r]\ar@{-->}[d]&\bullet\ar@{-}[r]\ar@{-->}[d]&\bullet\ar@{-->}[r]\ar@{-->}[d]   &\\
&&&
}   
\\  
\xymatrix{ \; \\
\bullet\ar@{-->}[u]\\
\bullet\ar@{-}[u]\\
\bullet\ar@{-}[u]
}
\end{array}
\]
\caption{\label{fig:w3}The lattice  $(\overline\nats\times\overline\nats) \cup \nats$}
\end{figure}
\end{Ex}

We end this section with a note about canonical extensions.
\begin{Lem}
A complete lattice $L$ is completely representable if and only if it is doubly algebraic (a complete lattice is algebraic if every element can be written as a join of compact elements, a complete lattice is doubly algebraic if both it and its order dual are algebraic).
\end{Lem}

\begin{proof}
It is known, see e.g. \cite{CrawDil73}, that a lattice $L$ is doubly algebraic if and only if it is complete, completely distributive, and the completely join/meet-irreducibles $J^\infty(L)$  and $M^\infty(L)$ are join/meet-dense respectively. When $L$ is completely representable it inherits complete distributivity from its representation and, by proposition \ref{irred}, has the required density properties. Conversely, by the same proposition, the density properties of algebraicity and dual algebraicity are both sufficient conditions for complete representability.
\end{proof}  

\begin{Cor}
The canonical extension of any bounded distributive lattice is completely representable.
\end{Cor}

\begin{proof}
The canonical extension $L^\sigma$ of a bounded distributive lattice $L$ can be defined as a doubly algebraic lattice into which $L$ embeds in a certain way (see e.g. \cite[theorem~2.5]{GehJon04}  for details).
\end{proof}

\end{section}

\begin{section}{HSP, elementarity and pseudo-elementarity} \label{SECpseud}
Since a subalgebra of an atomic Boolean algebra need not be atomic we know that none of the classes in \eqref{dagger} is closed under subalgebras, and thus cannot be varieties, or even quasi-varieties. Similarly, given an atomic Boolean algebra $B$ we can define an equivalence relation $R$ on $B$ by $xRy\iff |\{a\in At(B)\colon a\leq x\}\triangle \{a\in At(B)\colon a\leq y\}|<|\omega|$, that is, if and only if the symmetric difference of the sets of atoms beneath each element is finite. It can easily be shown that $R$ is a congruence, and in the case where $B$ is the complete, atomic Boolean algebra on $\omega$ generators the resulting $\frac{B}{R}$ is isomorphic to the countable atomless Boolean algebra, and thus none of classes in \eqref{dagger} can be closed under homomorphic images. We can say something positive about closure under direct products, which we express in the following lemma:
\begin{Lem}
The classes in \eqref{dagger} are all closed under taking direct products.
\end{Lem}
\begin{proof}
We do the proof for $\mCRL$, the others are similar. Suppose $\{L_i\}_I$ is a family of lattices in $\mCRL$. Let $f\neq g \in \prod_I L_i$, then we can choose $j\in I$ with $f(j)\neq g(j)$, and by the assumption of meet-complete representability there is a complete, prime filter $\gamma$ distinguishing $f(j)$ and $g(i)$. Define sets $S_i\subseteq L_i$ by $S_j=\gamma$ and $S_i=L_i$ for all $i\neq j$, then $S=\prod_I S_i$ is a complete, prime filter distinguishing $f$ and $g$. 
\end{proof}
 
As `being atomic' is a first order property for Boolean algebras, it follows immediately from corollary \ref{BooCor} that the class of completely representable Boolean algebras is elementary. The aim here is to investigate to what extent similar results hold for the classes in \eqref{dagger}. Our first result is negative:

\begin{Thm}\label{not elem}
\textbf{CRL} is not closed under elementary equivalence.
\end{Thm}
\begin{proof}
The lattice $L=[0,1]\subseteq \R$ from example \ref{exI} is not in \textbf{CRL}, however the lattice $L'=[0,1]\cap \Q$ \emph{is} in $\CRL$ as for every irrational $r$ the set $\{ a\in L'\colon a> r\}$ is a complete, completely-prime filter. $L$ and $L'$ are elementarily equivalent as $\R$ and $\Q$ are.
\end{proof}
We can, however, show that all the classes in \eqref{dagger} are at least \emph{pseudo-elementary}. In particular we shall demonstrate that $\mCRL$ is precisely the first order reduct of the class of models of a theory in two-sorted FOL, and thus is pseudo-elementary (the proof can be readily adapted for the other classes). We proceed as follows.

\begin{Def}\label{PEclass}\emph{(Pseudo-elementary class)} Given a first order signature $\sL$, a class $\sC$ of $\sL$ structures is pseudo-elementary if there are
\begin{enumerate}
\item a two-sorted language $\sL^+$, with disjoint sorts $\bL$ and $\bS$, containing $\bL$-sorted copies of all symbols of $\sL$, and
\item an $\sL^+$ theory $U$  
\end{enumerate}
with $\sC=\{M^{\bL}\upharpoonright_{\sL}\colon M\models U \}$, where $M^{\bL}$ is a structure in the sublanguage of $\sL^+$ containing only $\bL$-sorted symbols whose domain contains only $\bL$-sorted elements, $M^{\bL}\upharpoonright_{\sL}$ is the $\sL$ reduct of $M^{\bL}$ obtained easily by identifying the symbols of $\sL$ with their $\bL$-sorted counterparts in $\sL^+$, and $\{M^{\bL}\upharpoonright_{\sL}\colon M\models U \}$ being thus the class of $\bL$-sort $\sL$ reducts of models of $U$. 
\end{Def}
See \cite[section~9]{HirHod02}  for more information on this definition, and for proof of its equivalence with single-sorted definitions of pseudo-elementarity. 

Now, let $\sL=\{+,\cdot,0,1\}$ be the language of bounded, distributive lattices in FOL. Define the two-sorted language $\sL^+=\sL\cup\{\in\}$, where $\in$ is a binary predicate whose first argument takes variables of the $\bL$ sort and whose second takes variables of the $\bS$ sort. Let the original functions of $\sL$ be wholly $\bL$-sorted in $\sL^+$ (the $\bL$ sort is meant to represent lattice elements and the $\bS$ sort sets of these elements). Define binary $\bL$-sorted predicates $\leq$ and $\geq$, and binary $\bS$-sorted predicate $\subseteq$ in $\sL^+$ in the obvious way. For simplicity we will write $\xL\in \sS$ for $\in(\xL,\sS)$, and similar for $\leq$, $\subseteq$ etc.\\ 

Define additional predicates $P$, $I$ and $C$ as follows:
\begin{itemize}
\item
 $P(\sS)$ if and only if each of the following properties hold:
\begin{enumerate}
\item $\forall \xL \yL\Big(\big((\xL\in\sS)\wedge (\yL\geq \xL)\big)\rightarrow (\yL\in\sS)\Big)$
\item $\forall\xL\yL\Big(\big((\xL\in\sS) \wedge (\yL\in\sS)\big)\rightarrow (\xL\cdot \yL \in \sS)\Big)$
\item $\forall\xL\yL\Big((\xL+\yL\in\sS)\rightarrow \big((\xL\in \sS)\vee (\yL\in\sS)\big)\Big) $
\end{enumerate}
$P$ is meant to capture the property of being a prime filter.\\
\item
 $I(\xL,\sS)$ if and only if 
\[\forall \yL\Big((\yL \in \sS) \rightarrow (\xL\leq\yL)\Big)\wedge \forall\zL \Big(\big((\yL \in \sS) \rightarrow (\zL\leq\yL)\big)\rightarrow (\zL\leq\xL) \Big).\] $I$ corresponds to the notion of an element being the infimum of a set.
\item
 $C(\sS)$ if and only if $\forall \tS \forall\xL\Big( \big( (\tS \subseteq \sS)\wedge I(\xL,\tS)\big)\rightarrow (\xL\in \sS) \Big)$, so $C$ specifies a limited form of completeness.
\end{itemize}
Now, let $T$ be the $\sL$ theory of bounded, distributive lattices. Define $T^+$ as the natural translation of $T$ into the language $\sL^+$ plus the following additional axioms:

\renewcommand{\theenumi}{\Roman{enumi}}
\begin{enumerate}
\item \label{ax:1}$\forall \xL\yL \Big(\xL\neq\yL \rightarrow \exists\sS\Big( \big(P(\sS)\wedge C(\sS) \big) \wedge \big(\big((\xL\in\sS) \wedge (\yL\notin \sS)\big)\vee \big((\yL\in\sS) \wedge (\xL\notin \sS)\big)\big)  \Big)\Big)$
\item \label{ax:2}$\forall \xL \exists \sS \forall \yL \Big( (\yL>\xL)\leftrightarrow (y\in \sS) \Big)$
\item \label{ax:3} $\forall \sS\tS\exists \uS\forall\xL\Big( \big((\xL\in\sS)\wedge (\xL\in\tS)\big) \leftrightarrow (\xL\in \uS) \Big)$
\end{enumerate}
\renewcommand{\theenumi}{\arabic{enumi}}
The first of these axioms forces the $\bS$ sort into providing a distinguishing set of `complete' (with respect to $\bS$) prime filters, and the second and third force the existence of sufficiently many elements of $\bS$ that this notion of completeness is equivalent to actual completeness, as the lemma below demonstrates.

\begin{Lem}\label{PE}
The class $\{M^{\bL}\upharpoonright_{\sL}\colon M\models T^+ \}$ of $\sL$-reducts of models of $T^+$ is precisely the class of meet-completely representable bounded, distributive lattices.
\end{Lem}
\begin{proof}
Clearly if $L$ is in $\mCRL$ its elements satisfy $T$, and $(L, \sP(L), \in)$  satisfy $T^+$, where $\in$ is ordinary set membership. Conversely, if $A=M^{\bL}\upharpoonright_{\sL}$ for some model $M$ of $T^+$ then by axiom \ref{ax:1} of $T^+$ the (interpretation of) the $\in$ predicate naturally defines a distinguishing set $K$ of prime filters of $A$.  We claim that each prime filter in $K$ is complete.
For the claim, let $p\in K$ and $s\subseteq p$ with $x=$inf$(s)$.  We must show that $x\in p$.  If $x\in s$ then this is immediate, so we suppose not: $x\not\in s$.  We
 consider the following cases:
\begin{enumerate}
\item
 $x=$ inf$\{y\colon y>x\}$: then $s\subseteq \{y\colon y>x\}\cap p \subseteq \{y\colon y>x\}$ so $x=$ inf$(s)\geq $ inf$(\{y\colon y>x\}\cap p)\geq x$ and thus inf$(\{y\colon y>x\}\cap p)=x$, but clearly $\{y\colon y>x\}\cap p \subseteq p$ and by axioms~\ref{ax:2} and \ref{ax:3} of $T^+$ also corresponds to an element of the $\bS$ sort. Therefore, by definition of the predicate $C$ we have $x\in p$, as required.

\item $x\neq $ inf$\{y\colon y>x\}$: Let $z$ be a lower bound for $\{y\colon y>x\}$, suppose $z\not \leq x$. Then $x\vee z$ is a lower bound for $\{y\colon y>x\}$ and is contained in $\{y\colon y>x\}$. In light of this assume wlog that inf$\{y\colon y>x\}= z>x$.  Then, as $x=$ inf$(s)$, we have $s\subseteq \{y\colon y>x\}$ and thus $s$ has $z$ as a lower bound, but this a contradiction as $x<z$, so this case cannot arise.
\end{enumerate}
We deduce that $x\in p$, so $p$ is complete, as claimed.
Since $T^+$ demands $A$ be a bounded, distributive lattice we have $A\in\mCRL$, by theorem~\ref{main2}(1).
\end{proof}

We have now proved the following:
\begin{Thm}\label{pseudelem}
$\mCRL$ is pseudo-elementary.
\end{Thm}
\begin{proof}
Lemma \ref{PE} shows the condition of definition \ref{PEclass} hold for $\mCRL$.
\end{proof}
It is not difficult to see how analogous results can also be proved for $\jCRL$, $\biCRL$ and $\CRL$ using a similar method.
 
\end{section}

\begin{section}{Ultraproducts and ultraroots} \label{SECur}
We begin this section by stating two well known facts from model theory:
\begin{Fact}\label{she cor}
A class $\mathscr{C}$ of similar structures is elementary iff it is closed under isomorphism, ultraproducts and ultraroots.
\end{Fact}
\begin{Fact}
$\sC$ is pseudo-elementary $\implies$ $\sC$ is closed under ultraproducts.
\end{Fact}
For proof of the first see e.g. \cite[corollary~9.5.10]{Hodg93} or \cite[theorem~4.1.12 and corollary~4.3.13]{ChaKei90}. The proof of the second is simple, see e.g. \cite[exercise~4.1.17]{ChaKei90}.

In view of the facts above and the material in the preceding section, since $\CRL$ is pseudo-elementary, and closed under isomorphism, but is not elementary, it cannot be closed under ultraroots. $\mCRL$, $\jCRL$ and $\biCRL$ will be elementary if and only if they are closed under ultraroots. Note that $\mCRL$ is elementary iff $\jCRL$ is elementary (by duality), and therefore $\mCRL$ is elementary $\implies\biCRL$ is elementary (as $\biCRL=\mCRL\cap\jCRL$). It is not known which, if any, of $\biCRL$, $\mCRL$ and $\jCRL$ are closed under ultraroots but it is possible to state some conditions on a lattice $L$ which must necessarily hold if $L\not\in X$ but an ultrapower of $L$ belongs to $X$ (where $X=\biCRL, \mCRL$ or $\jCRL$).

First of all in order for the ultraproduct $\UP$ to be meet-completely representable $L$ must be $\vee(\bw)$-distributive, i.e. for $a\in L,\; S\subseteq L$ if both sides of the equation below are defined then they are equal
\[ a\vee\bw S =\bw_{s\in S}(a\vee s)\]
as we shall see in the next proposition. Note that the converse to this is false as, for example, every Boolean algebra is $\vee(\bw)$-distributive (see e.g. \cite[theorem~5.13]{Rom08}  for a proof) but not necessarily atomic, so not necessarily meet-completely representable by corollary \ref{BooCor}. We will use the following notation and lemma:
\begin{itemize}
\item For $a \in L$ define $\bar{a}\in \prod_I L$ by $\bar{a}(i)=a$ for all $i\in I$.
\item  Fix some ultrafilter $U$ over $I$.  For $x\in\prod_IL$ we write $[x]$ for $\set{y\in\prod_IL\colon \colon \set{i\colon x(i)=y(i)}\in U}$.
\item For $S\subseteq L$ define $S^*=\{[x]\in \UP \colon \{i\in I\colon x(i)\in S\}\in U\}$.
\item For $T\subseteq\UP$ define $T_*=\{a\in L\colon [\ba]\in T\}$.
\end{itemize}

\begin{Lem}\label{inf exist}
Let $S\subseteq L$ and suppose $\bw S$ exists in $L$. Then $\bw (S^*)$ exists in $\UP$ and equals $[\overline{\bw S}]$. 
\end{Lem}

\begin{proof}
This can be proved by defining an additional predicate `$S$' in the language of lattices meant to correspond to `being an element of the set $S$', the result then following easily from \L o\'s' theorem. An alternative algebraic proof is as follows:
Clearly $[\overline{\bw S}]$ is a lower bound for $S^*$. Suppose $[z]$ is another such lower bound and $[z]\not \leq [\overline{\bw S}]$. Then $\{i\in I\colon z(i)\not \leq \bw S\}\in U$, so $\{i\in I\colon \exists s_i\in S$ with $z(i)\not \leq s_i\}\in U$, $=u$ say (as $\bw S$ is the greatest lower bound of $S$). Define $x$ by $x(i)=s_i$ for $i\in u$ and $x(i)=\bw S$ otherwise. Then $[x]\in S^*$ but $[z]\not \leq [x]$, but this contradicts the assertion that $[z]$ is a lower bound. 
\end{proof}  

\begin{Cor}\label{dist root}
The class of $\vee(\bw)$-distributive bounded lattices is closed under ultraroots.
\end{Cor}
\begin{proof}
By lemma \ref{inf exist} if there is some $A\cup\{b\}\subseteq L$ with $b\vee\bw A\neq\bw(b\vee A)$ then $\bw A^* \vee [\bb]=[\bar {\bw A}]\vee [\bb]=[(\overline{\bw A) \vee} b]\neq [\overline{\bw (A\vee b})]=\bw (A^*\vee [\bb])$, so if $L$ is not $\vee(\bw)$-distributive then neither is $\UP$.
\end{proof}

\begin{Prop}\label{ur1}
If $\UP$  has a meet-complete representation then $L$ is $\vee(\bw)$-distributive. 
\end{Prop}
\begin{proof}
This follows from corollary \ref{dist root} and the fact that when $\UP$ is in $\mCRL$ it inherits $\vee(\bw)$-distributivity from its representation. 
\end{proof}

By duality a similar result holds for $\jCRL$, and hence for $\biCRL$. In order for $\UP$ to be in $\mCRL$ but $L$ not to be it turns out $L$ must satisfy an infinite density property, which we make precise in the next proposition.
\begin{Prop}\label{ur2}
If $\UP$  has a meet-complete representation but $L$ does not then there is a pair $x<y$ such that for every pair $a<b\in [x,y]$ there is some $c$ with $a<c<b$. 
\end{Prop}

\begin{proof}
If $L$ is not in $\mCRL$ then there is a pair $x,y,\in L$ that cannot be distinguished by a complete, prime filter. Wlog assume $x<y$. Since $\UP$ is in $\mCRL$, for each pair $a<b\in [x,y]$ there is a complete, prime filter $\gamma$ distinguishing $[\ba]$ and $[\bb]$. It's easy to show that $\gamma_*$ is a prime filter of $L$ with $b\in \gamma_*$ and $a\notin \gamma_*$ (and thus $y\in \gamma_*$ and $x\notin \gamma_*$). Let $a<b$ and $(a,b)=\emptyset$ and suppose $S\subseteq \gamma_*$. Then for each $[z]\in S^*$ we have must have $[z]\vee[\ba]=[\bb]$, and thus by primality $S^*\subseteq \gamma$. So by lemma \ref{inf exist} $\bw S\in \gamma_*$, and so $\gamma_*$ is complete, which is a contradiction as we assumed $x$ and $y$ could not be distinguished by a complete, prime filter.
\end{proof}
Again by duality the same result holds for join-complete representations. Note that if we could find a counter example $(L, \prod_U L)$ where $\UP\in\mCRL, \;$ $L\not\in\mCRL$,  we could restrict to the sublattice bounded by $x$ and $y$, so we lose nothing by assuming that $x$ and $y$ are the lower and upper bounds respectively, and that the whole lattice therefore has this density property. 
 
We have seen that the class of completely representable Boolean algebras is atomic (indeed finitely axiomatisable) and that the class $\CRL$ of completely representable lattices is not.
\begin{Conj}
None of the classes $\jCRL, \mCRL, \biCRL$ is elementary.
\end{Conj}
We intend to prove this in a subsequent article.
\end{section}

\end{document}